\documentclass[12pt,twoside]{article}
\usepackage{latexsym, amssymb,amsthm,amsmath,xypic,lscape,epsfig,setspace,tikz,slashed}
\usepackage[retainorgcmds]{IEEEtrantools}
\usetikzlibrary{decorations.markings}
 \usetikzlibrary{arrows}
\usepackage[pdftex]{hyperref}
\usepackage[margin=10pt,font=small,labelfont=bf]{caption}

\theoremstyle{plain}
\newtheorem{theorem}{Theorem}[section]
\newtheorem{lemma}[theorem]{Lemma}%[section]
\theoremstyle{definition}
\newtheorem{definition}[theorem]{Definition}
\AtBeginEnvironment{definition}{%
  \pushQED{\qed}%
}
\AtEndEnvironment{definition}{\popQED\endexample}

\newtheorem{corollary}[theorem]{Corollary}%[section]

\newtheorem{example}[theorem]{Example}
\AtBeginEnvironment{example}{%
  \pushQED{\qed}%
}
\AtEndEnvironment{example}{\popQED\endexample}

\newenvironment{renumerate}%
{%
\begin{enumerate}}%
{\end{enumerate}%
}%

\newenvironment{remark}%
{\vskip6pt%
\noindent%
{\it Remark.}}%
{\vskip6pt}

{\vskip6pt%
\noindent%
{\it Remarks}. %
\begin{renumerate}}%
{\end{renumerate}\vskip6pt}

\makeatletter
\def\Ddots{\mathinner{\mkern1mu\raise\p@
\vbox{\kern7\p@\hbox{.}}\mkern2mu
\raise4\p@\hbox{.}\mkern2mu\raise7\p@\hbox{.}\mkern1mu}}
\makeatother

\newcommand{\Z}{\text{$\mathbb Z$}}

\renewcommand{\tilde}{\widetilde}

\newcommand{\A}{\text{$\mathcal{A}$}}

\newcommand{\U}{\mathrm{U}}
\newcommand{\Clif}{\mathrm{Clif}}

\newcommand{\gf}{\text{$\varphi$}}

\newcommand{\Id}{\mathrm{Id}}

\newcommand{\del}{\text{$\partial$}}

\newcommand{\tensor}{\otimes}

\newcommand{\End}{\mathrm{End}\,}

\newcommand{\mc}[1]{\text{$\mathcal{#1}$}}

\newcommand{\noqed}{\let\qed\relax}

\newcommand{\IP}[1]{\langle #1 \rangle}

\newcommand{\Cour}[1]{[\![#1]\!]}

\date{} \usepackage{color} \definecolor{tocolor}{rgb}{.1,.1,.5}
\definecolor{urlcolor}{rgb}{.2,.2,.6}
\definecolor{linkcolor}{rgb}{.1,.1,.6}
\definecolor{citecolor}{rgb}{.6,.2,.1}
\hypersetup{backref=true, colorlinks=true, urlcolor=urlcolor,
  linkcolor=linkcolor, citecolor=citecolor}

\begingroup
\hypersetup{linkcolor=blue}

\endgroup

\numberwithin{equation}{section}

\begin{document}

\title{T-duality for transgressive fibrations}
\author{Gil R. Cavalcanti\thanks{{\tt gil.cavalcanti@gmail.com}, Department of Mathematics, Universiteit Utrecht, Utrecht, The Netherlands.}}
\maketitle

\abstract{We extend the notion of topological T-duality from oriented sphere bundles to transgressive fibrations, a more general type fibration characterised by the abundance of transgressive elements. Examples of transgressive fibrations include principal $\mathrm{U}(n)$-bundles therefore our notion of T-duality belongs to the realm of non-Abelian T-duality. We prove that transgressive T-duals have isomorphic twisted cohomology. We then introduce Clifford--Courant algebroids, show that one can assign such an algebroid to a transgressive fibration and that transgressive T-duals have isomorphic Clifford--Courant algebroids. We provide several examples illustrating different properties of T-dual spaces.}
\vskip12pt
\noindent
MSC classification 2020: 	55R25, 57R22, 57R19, 57N65.\\
Subject classification: Differential Geometry, Differential Topology.

\tableofcontents

\section{Introduction}

Target space duality, or just T-duality, is a concept that comes from physics. In its original form, it concerns string theory when there is an $S^1$-symmetry in the target space. In this case, there is another model for an equivalent physical theory in a different space called the {\it T-dual}. This idea was first worked out in 87 by Buscher \cite{Bus87} and became so central in string theory that in general one simply expects new physical theories to be invariant under T-duality in one form or another. 

Because the circle is an Abelian group, this original T-duality is often referred to as Abelian T-duality and there is a lot of interest in understanding different aspects of this duality. Yet models with a non-Abelian group of symmetries arise frequently and simply restricting to a maximal torus overlooks significantly the full symmetry of the situation. Many attempts have been made to produce a general framework for non-Abelian T-duality \cite{DELAOSSA1993377,Giveon:1993ai,KlSe95,VonUnge:2002xjf} and this quest remains a relevant research question today.

Mathematically the notion of T-duality was encoded in two different approaches: Poisson-Lie duality \cite{KlSe95} and topological T-duality \cite{BEM04}. The Poisson--Lie framework seems to match well the physics it proposes to model and does not require Abelian group actions. Its down side is that it is not constructive. Topological T-duality also matches well the physical theories in the case of an Abelian group of symmetries and was very successful because it  is constructive and the construction makes clear the interplay between string fluxes and the topology of T-dual spaces.

Hallmarks of topological T-duality are that T-dual spaces have isomorphic twisted cohomologies, twisted K-theories \cite{BEM04} and associated Courant algebroids \cite{CG11}. The first two properties are not shared with Poisson--Lie duality. A different extension of topological T-duality, known as {\it spherical T-duality}, which removes the requirement of Abelian symmetries was proposed in \cite{BEM15} and further developed in \cite{BEM15b,BEM18,LSW16,CHU24}. Despite missing even the existence of a symmetry group, spherical T-dual spaces have  isomorphic twisted cohomologies, higher twisted K-theories and associated extended higher Courant algebroids. These results make spherical T-duality a mathematically relevant relation between spaces both from the algebraic topological and differential geometric view points. Further,  there are indications that specific examples of spherical T-duals are relevant in string theory.

Arguably the place where spherical T-duality first diverges from Poisson--Lie is that in Poisson--Lie the symmetries of T-dual spaces are dual groups, while spherical T-dual spaces have fibers with dual cohomology. This is the point of view we will take and extend in this paper. We will introduce the notion of transgressive fibrations (Definition \ref{def:transgressive fibration}) which include fibrations whose fibers can be more interesting than spheres and in fact capture important examples such as  the full or any partial frame bundle of Hermitian vector bundle. We introduce the notion of T-duality for transgressive fibrations (Definition \ref{def:T-duality}) and prove that once again one has isomorphisms of twisted cohomology (Theorem \ref{theo:iso1}). We then introduce the Clifford--Courant algebroids, a Courant algebroid-like object compatible with this version of T-duality. After studying the most basic properties of Clifford--Courant algebroids we prove that T-dual spaces have isomorphic algebroids (Theorem \ref{theo:iso2}). We then move on to give several examples of T-dual spaces, which show that the definition leads to concrete constructions.

\section{Transgressive bundles}

In this section we introduce the objects to which we will extend T-duality, namely transgressive fibrations. Throughout the paper, for simplicity we assume that our spaces are smooth manifolds and work with de Rham cohomology, though much of the theory can be carried out for topological manifolds and integral singular cochains. First we recall the definition of a transgressive cohomology class.

\begin{definition}
Given a fibration $F \cdots E \stackrel{\pi}{\longrightarrow}M$, a cohomology class $A \in H^\bullet(F)$ is {\it transgressive} if  there is a form $\psi \in \Omega^\bullet(E)$ such that
\begin{itemize}
\item $[\iota^*\psi] =A$ (where $\iota\colon F \to E$ is the inclusion of a reference fiber) and
\item $d\psi = \pi^* c$ for some $c \in \Omega^\bullet(M)$.
\end{itemize}
\end{definition}

This leads directly to the concept of transgressive fibration:

\begin{definition}\label{def:transgressive fibration}
A fibration $F\cdots E\stackrel{\pi}{\longrightarrow} M$ is \emph{transgressive} if the cohomology algebra $H^\bullet(F)$ is generated by transgressive elements.
An \emph{odd transgressive fibration} is a transgressive fibration whose fibers have the same cohomology algebra of a product of odd spheres.
\end{definition}

Odd transgressive fibrations are arguably very exceptional and one might wonder how frequently they appear. Fortunately, they arise naturally as fibrations associated to vector bundles.

\begin{example}
Given a Hermitian vector bundle $V \to M$ of rank $n$, let $E_k \stackrel{\pi}{\longrightarrow} M$ be the bundle of unitary $k$-vectors on $V$. Then the associated $\U(n)$-bundle of $V$ acts on $E_k$ with isotropy $\U(n-k)$, so $E_k$ is a fibration
\[\U(n)/\U(n-k)\cdots E_k \to M.\]
The cohomology of  $\U(n)/\U(n-k)$ is generated by single generators in degrees $2(n-k)+1, \dots, 2n-1$ and there are forms $\psi_{2j-1}\in \Omega^{2j-1}(E_k)$ whose restriction to the fiber represents a  generator of $H^{2j-1}(F)$ and such that $d \psi_{2j-1} = \pi^*c_{j}$, where $c_{j}$ is a representative of the $j^{th}$ Chern class of $V$.

It is worth mentioning the extremal cases. The bundle $E_1 \to M$ is the sphere bundle of $E$ and $\psi_{2n-1}$ is often called a {\it global angular form}  of $E_1$. The bundle $E_n\stackrel{\pi}{\longrightarrow}{M}$ is the frame bundle of $V$ and the existence of the forms $\psi_{2j-1}$ above is a manifestation of the fact that $\pi^*V \to E_n$ is the trivial bundle, hence all the Chern classes of $V$ become exact when pulled back to the frame bundle.

For the general case, the existence of the forms $\psi_{2j-1}$ above are related to the interpretation of the Chern classes  $c_{n-k+1}, \cdots, c_n$ as obstructions to the existence of a global $k$-unitary set on $V$ (a set which exists tautologically on $\pi^*V \to E_k$).
\end{example}

\begin{definition}
Let $E \to M$ be an odd transgressive fibration. A {\it transgressive generating set} is a collection of forms $\psi_{2i-1}^j \in \Omega^{2i-1}(E)$ such that
\begin{itemize}
\item $\Psi = \{ \psi_{2i-1}^j:j=1,\dots, n_i, \deg(\psi_{2i-1}^j) = 2i-1 \}$ freely generates the integral cohomology of the fiber
\item $d\psi_{2i-1}^j = \pi^*c_{2i}^j$.
\end{itemize}
The complex
\[\Omega_{\Psi}^\bullet(M) =  \wedge^{\bullet}\IP{\Psi}\tensor\Omega^\bullet(M) \subset \Omega^{\bullet}(E)\]
is the {\it transgressive subcomplex}.
\end{definition}

A transgressive subcomplex already contains all cohomological information of $E$:

\begin{lemma}\label{lem:quasi-iso1}
Given an odd transgressive fibration $F\cdots E \to M$ with a transgressive generating set, $\Psi$, the exterior derivative preserves $\Omega_\Psi$ and the inclusion
\[ (\Omega_\Psi,d) \to (\Omega(E),d)\]
is a quasi-isomorphism. 
\end{lemma}
\begin{proof}
It follows from the definition of transgressive generating set that $\Omega_\Psi $ is preserved by the exterior derivative and hence is indeed a subcomplex. Further, the existence of a transgressive generating set implies that the fundamental group of $M$ acts trivially on the cohomology of $F$.  The cohomology of  $(\Omega(E),d)$ is then computed by the Serre spectral sequence whose first page is isomorphic to $\Omega_{\Psi}$. A similar spectral sequence computes the cohomology of $(\Omega_\Psi,d)$ and we have that the inclusion $(\Omega_\Psi,d) \to (\Omega(E),d)$ induces an isomorphism of the their first pages, hence these two complexes have the same cohomology.
\end{proof}

In this paper we will be interested in the twisted cohomology of spaces, that is the cohomology of $d+H\wedge$, where $H$ is a closed odd form, possibly of mixed degree, with lowest degree component of degree at least 3. Twisted cohomology only depends on the cohomology class of $H$, so, in the presence of a transgressive generating set, due to the previous lemma, we may assume that the closed form in question lies in $\Omega_\Psi$. It then turns out that the twisted cohomology can also be computed directly in the transgressive subcomplex.

\begin{lemma}
Given an odd transgressive fibration $F\cdots E \to M$ with a transgressive generating set, $\Psi$, and a closed odd element $H \in \Omega_{\Psi} \subset \Omega(E)$, the twisted exterior derivative preserves $\Omega_\Psi$ and the inclusion
\[ (\Omega_\Psi,d^H) \to (\Omega(E),d^H)\]
is a quasi-isomorphism. 
\end{lemma}
\begin{proof}
The decomposition $d^H = d + H\wedge$ with $d^2 =0$ and $H^2=0$ makes both $\Omega_\Psi$ and $\Omega(E)$ into filtered double complexes whose cohomology can then be calculated via a spectral sequence. The first nontrivlal page of these spectral sequences is the usual cohomology and the differential in the higher pages are related to Massey products of the form $\IP{H,\dots,H,\bullet}$. Since by Lemma \ref{lem:quasi-iso1} the inclusion $(\Omega_\Psi,d) \to (\Omega(E),d)$ induces an isomorphism in cohomology, the inclusion $(\Omega_\Psi,d^H) \to (\Omega(E),d^H)$ induces an isomorphism of their first pages and hence an isomorphism of their last page.
\end{proof}

We finish this section on transgressive fibrations with a simple observation
\begin{lemma}
Let $E \to M$ and $\hat{E}\to M$ be two odd transgressive fibrations and $\Psi$ and $\hat{\Psi}$ be transgressive generating sets. Then $E \times _M \hat{E} \to M$ is an odd transgressive fibration with transgressive generating set  $p^*\Psi \cup \hat{p}^*\hat\Psi$, where $p$ and $\hat{p}$ are the natural projections to $E$ and $\hat{E}$ respectively.
\end{lemma}

\section{T-duality for transgressive bundles}

\subsection{Definition of T-duality}

In this section we introduce the notion of T-duality for odd transgressive fibrations. To phrase the notion of nondegeneracy we need one final concept.

\begin{definition}
Let $E\to M$ be an odd transgressive fibration and $\Psi$ a transgressive generating set. An element in $\Omega_\Psi$ has \emph{basic-degree $k$} if it lies in  $\wedge^\bullet\IP{\Psi}\tensor \Omega^k(M)$.   
\end{definition}
Notice that the exterior derivative  is compatible with the filtration by basic-degree:
\[d \colon\wedge^\bullet\IP{\Psi} \tensor  \Omega^k(M)  \to\wedge^\bullet\IP{\Psi} \tensor ( \oplus_{j>k}\Omega^j(M)) ,\]
and if we project to the basic-degree $k+1$ component, 
\[ \Pi_{k+1}\circ d \colon \wedge^\bullet\IP{\Psi} \tensor  \Omega^k(M) \to \wedge^\bullet\IP{\Psi} \tensor  \Omega^{k+1}(M),\]
the effect is to just take the exterior derivative of the coefficients.

\begin{definition}\label{def:T-duality}
Two odd transgressive fibrations $(E,H)\stackrel{\pi}{\longrightarrow} M$ and $(\hat{E},\hat{H})\stackrel{\hat\pi}{\longrightarrow} M$ with transgressive generating sets $\Psi$ and $\hat\Psi$ are \emph{T-dual} if there is a form $F \in \Omega_{\Psi\cup\hat{\Psi}}$ such that
\begin{enumerate}
\item (gerbe trivialization)
\[dF = p^*H - \hat{p}^*\hat{H}\]
\item (nondegeneracy)\[\tau_{F_0}:= \hat{p}_*e^{F_0}p^*\colon \wedge \IP{\Psi} \tensor \Omega^0(M)   \to  \wedge \IP{\hat\Psi}\tensor \Omega^0(M)\]
 is a fiberwise isomorphism,
\end{enumerate}
where $p\colon E \times _M\hat{E}\to E$ and $\hat{p}\colon E \times _M\hat{E}\to \hat{E}$ are the natural projections and $F_0$ is the basic-degree zero component of $F$. The form $F$ is called the \emph{T-duality kernel}.
\end{definition}
The definition of T-dual spaces above seems to depend on the particular choices of  transgressive generating sets. This is not the case.
\begin{lemma}
If $(E,H)\stackrel{\pi}{\longrightarrow} M$ and $(\hat{E},\hat{H})\stackrel{\hat\pi}{\longrightarrow} M$ are T-dual for one choice of transgressive generating sets, then they are T-dual for any choice generating the same transgressive complex.
\end{lemma}
\begin{proof}
To be able to refer to the spaces in question in a base-free way, we let  $\A = \wedge^\bullet\IP{\Psi}$ and $\hat\A = \wedge^\bullet\IP{\hat\Psi}$.

First  let us go through some algebraic preliminaries. For a choice of generating set
\[\Psi = \{ \psi_{2i-1}^j:j=1,\dots, n_i, \deg(\psi_{2i-1}^j) = 2i-1 \}\]
let $\sigma_\Psi = \wedge_{i}\wedge_{j=1}^{n_i}\psi_{2i-1}^j$ be the volume element of that generating set. Notice that for a different choice of generating set, say $\tilde\Psi$, we have, for each $i$ and $j$ that 
\[\tilde\psi_{2i-1}^j = \sum_{k}A^j_{2i-1,k}\psi_{2i-1}^k + p^j_{2i-1},\]
where $p^j_{2i-1}$ is an element in the exterior algebra generated by the generators of degree less than $2i-1$. In particular the matrix $A_{2i-1} = (A^j_{2i-1,k})$ is invertible and we have that
\begin{equation}
\sigma_{\tilde{\Psi}} = \Pi_i \det(A_{2i-1})\sigma_\Psi.
\end{equation}
Also, for the same reason, writing the wedge product of a partial collection of the elements $\tilde\psi$ in the basis $\Psi$ does not contain the element $\sigma_{\Psi}$ as a component. 

Rephrasing this in a more elegant way, if $|\Psi| = N$ (that is $\A$ is a free exterior algebra on $N$ generators), and  $\A^{<N} $ are the elements of $\A$ obtained as linear combinations of the wedge product of at most $N-1$ generators, then $\mc{A}^{<N}$ is a codimension-one subspace,
\[0 \to \A^{<N} {\longrightarrow}\A \stackrel{\Pi_{top}}{\longrightarrow} \A/ \A^{<N} \to 0\]
is a short exact sequence and $\Pi_{top}(\sigma_\Psi) \neq 0$.

Now to the proof. Firstly, it is clear that changing the transgressive generating set $\hat\Psi$ changes the map $\tau_{F_0}$ by an overall change of basis for the vector space $\hat\A$ which does not affect the isomorphism condition, so we only need to study how changing the generating set $\Psi$ affects $\tau_{F_0}$.

Secondly, we can describe the map $\tau_{F_0}$ in more general terms as the composition of a product  (with the form $e^{F_0}$) and a projection onto $\hat\A \tensor \sigma_\Psi$. The product in $\A$ is a well defined operation and independent of basis and the projection is (equivalent to) $\Id_{\hat{\A}}\tensor \Pi_N$, which is independent of basis. Hence being an isomorphism is independent of  the generating set.
\end{proof}

\begin{lemma}\label{lem:basic-degree zero part}
For a T-dual pair, the basic-degree zero component of the T-duality kernel, $F$, satisfies
\[F_0 = \pi^*F_E + \hat\pi^*F_{\hat{E}} + \sum_{I,J\neq \emptyset}F_{IJ}\Psi_I\hat{\Psi}_J\]
with    $F_{IJ}$ constant and $F_E$ and $F_{\hat{E}}$ basic-degree zero forms in the transgressive complexes of $E$ and $\hat{E}$. In particular if $M$ is connected, the nondegeneracy condition only needs to be checked at a single fiber and is equivalent to the map
\[\hat{p}_*e^{\sum_{I,J\neq \emptyset}F_{IJ}\Psi_I\hat{\Psi}_J}p^*\colon \wedge \IP{\Psi} \to \wedge \IP{\hat\Psi}\]
being an isomorphism.
\end{lemma}
\begin{proof}
The mixed term of the basic-degree one component of $dF$ is  $\sum_{I,J\neq \emptyset} dF_{IJ}\Psi_I\hat{\Psi}_J $, while $p^*H - \hat{p}^*\hat{H}$ has no mixed terms, hence for all pairs of nonempty multi-indices $I$ and $J$,   $dF_{IJ}=0$ and the functions $F_{IJ}$ are constant on connected components of $M$.

Finally notice that the terms $\pi^*F_E$ and $\hat\pi^*F_{\hat{E}}$  do not affect the nondegeneracy condition since they correspond to pre and post-composing by automorphisms of $\Omega_\Psi$ and $\Omega_{\hat{\Psi}}$, so nondegeneracy is equivalent to the condition stated in the lemma.
\end{proof}

\begin{corollary}
In the notation of the previous lemma, if $\Psi = \{\psi_i\colon i=1,\dots n\}$, $\hat{\Psi} = \{\hat\psi_i\colon i=1,\dots n\}$ and the mixed component of $F_0$ is quadratic and nondegenerate, that is
\[F_0 = \pi^*F_E + \hat\pi^*F_{\hat{E}} +  \sum_{i,j=1}^n F_{ij} \psi_i \hat\psi_j\]
and the matrix $(F_{ij})$ is nondegenerate, then $\tau_{F_0}$  is an isomorphism.
\end{corollary}
In the usual T-duality for principal torus bundles  with background 3-form fluxes, due purely to degree consideratons, the situation decribed in the corollary above is the only case that can happen and hence the nondegeneracy condition is normally phrased by saying that $F$ puts the fibers of the two torus bundles in duality. Once we allow for higher degree twists, the T-duality kernel no longer needs to be only quadratic in the transgressive  generating set (in fact, being quadratic is not even independent of the choice of such set) and hence our definition of transgressive T-duality is in fact more general (see Example \ref{ex:identitiy?}).

\subsection{Isomorphism of twisted cohomology}

This definition of T-dual spaces yields the usual isomorphism of twisted cohomology between T-duals.

\begin{theorem}\label{theo:iso1}
Let $(E,H)\stackrel{\pi}{\longrightarrow} M$ and $(\hat{E},\hat{H})\stackrel{\hat\pi}{\longrightarrow} M$ be a T-dual pair with $\Psi$ and $\hat{\Psi}$ transgressive generating sets and $dF =p^*H - \hat{p}^*\hat{H}$.  Then the map
\[\tau_F\colon (\Omega_{\Psi},d^H) \to (\Omega_{\hat{\Psi}},d^{\hat{H}}),\quad  \tau_F = \hat{p}_*\circ e^F\circ p^* \]
is an isomorphism. In particular T-dual spaces have isomorphic twisted cohomology.
\end{theorem}
\begin{proof}
The gerbe trivialization condition, $dF = p^*H - \hat{p}^*\hat{H}$, yields that $\tau_F$ is a map of complexes:
\begin{align*}
\tau_{F}(d^H\gf) &= \hat{p}_*e^F p^* d^{H}\gf = \hat{p}_*d^{p^*H- dF}(e^F p^*\gf)\\
& =   \hat{p}_*d^{\hat{p^*}\hat{H}}(e^F p^*\gf) = d^{\hat{H}}\hat{p}_*e^F p^*\gf\\
&= d^{\hat{H}}\tau_F(\gf).
\end{align*}

Finally, since $\tau_F$ is a map of $\Omega^\bullet(M)$-modules, to check it is an isomorphism it is enough to check that the image of the generating set $\Psi$ is a generating set of $\wedge^\bullet \hat{\Psi}$, which in turn is the case if and only if the projection to the basic-degree zero part is a generating set. That is $\tau_F$ is an isomorphism if and only if
\[\Pi_0 \circ \tau_F|_{\Omega^0(M)\tensor \wedge\IP{\Psi}} \to \Omega^0(M)\tensor \wedge\IP{\hat\Psi}\]
is an isomorphism. But this is precisely the map that appears in the nondegeneracy condition for T-duality. 
\end{proof}

\begin{definition}
We refer to the map $\tau_F$ as the \emph{T-duality map on forms} or the \emph{T-duality transform}. 
\end{definition}
 
\subsection{Extensions of Courant algebroids}

Extended Courant algebroids, introduced in \cite{LUPERCIO201482}, proved to be relevant geometric objects for spherical T-duality \cite{CHU24}. These are obtained from higher Courant algebroids, that is, $TM \oplus \wedge^{2(n+k)-3}T^*M$, by adding a formal odd element, $\psi$ of degree $2n-1$ similar to a Hirsch extension of a differential graded algebra and its dual element $\del_\psi$ of degree $-2n+1$:
\begin{equation}\label{eq:extended higher}
c_\psi = TM \oplus (\wedge^{2n-2}T^*M\tensor\IP{\del_\psi}) \oplus (\IP{\psi}\tensor\wedge^{2k-2}T^*M)\oplus \wedge^{2(n+k)-3}T^*M 
\end{equation}
The closed forms allowed to twist the bracket are elements of $\Omega_\psi^{2(n+k)-1}(M)$.

In the present setting we will use the same general idea, but need to extend the higher Courant algebroid by a bigger vector space to capture enough features of odd transgressive fibrations. 

\subsubsection{Clifford--Courant algebroids}

The underlying algebraic object we will use is a Clifford algebra. Recall that given an odd-graded finite dimensional vector space, $V$, we can endow $V \oplus V^*$ with the symmetric natural pairing. The action of $V\oplus V^*$ on $\wedge^\bullet V^*$ by interior and exterior product extends to an action of the Clifford algebra $\Clif(V\oplus V^*)$ on $\wedge^\bullet V^*$ and gives rise to an algebra isomorphism $\Clif(V\oplus V^*) = \End(\wedge^\bullet V)$. In our case, we will take $V^*$ to be the space generated by a transgressive generating set: $V^* = \IP{\Psi}$.

Given a manifold $M$ and an odd-graded vector space, $V^*$, we can form the bundle
\[C_{V^*} = TM \oplus (\Clif(V\oplus V^*)\tensor \wedge^\bullet T^*M)^{od},\]
where ${od}$ indicates taking the odd degree component.

If we further have a degree-one linear map $d^V \colon V^* \to \Omega^\bullet(M)$ whose image lies in the space of closed forms, we can extend the exterior derivative from an operator on $\Omega^\bullet(M)$ to an operator on $\wedge^\bullet V^*\tensor \Omega^\bullet(M)$ by the Leibniz rule. The extended derivative on $\wedge^\bullet V^*\tensor \Omega^\bullet(M)$ is essentially $d^V + d$. We can describe $d^V$ explicitly: given a basis $\{\del_{\psi_1},\dots,\del_{\psi_N}\}$ for $V$ and dual basis $\{\psi_1,\cdots, \psi_N\}$ for $V^*$, we let $c_i = d^V \psi_i$ and then
\[d^V = \sum \del_{\psi_i} \tensor c_i  \in \Clif(V\oplus V^*)\tensor \Omega^\bullet(M).\]

Given a closed odd form $H \in \wedge^\bullet V^*\tensor \Omega^\bullet(M)$ we can form the twisted differential $d^H = d^V +d + H$.

\begin{lemma}
For $v,w\in \Gamma(TM \oplus (\Clif(V\oplus V^*)\tensor \wedge^\bullet T^*M)^{od})$ there is a unique element $\Cour{v,w}_H \in \Gamma(TM \oplus (\Clif(V\oplus V^*)\tensor \wedge^\bullet T^*M)^{od})$ for which the following holds:
\[\Cour{v,w}_H \cdot \gf = \{\{v,d^H\},w\} \gf,\]
where $\{\cdot,\cdot\}$ denotes the graded commutator of operators.
\end{lemma}
\begin{proof}
Both $d^V$ and $H$ are  fiberwise elements of $\Clif(V\oplus V^*)\tensor \wedge^\bullet T^*M$ hence taking graded commutators with them amounts to the commutator operation in the Clifford algebra. Therefore, to prove the result it is enough to check that  $\{\{v,d\},w\}$ corresponds to the action of an element in  $\Gamma(TM \oplus (\Clif(V\oplus V^*)\tensor \wedge^\bullet T^*M)^{od})$, where $d$ is the exterior derivative in $\Omega^\bullet(M)$ which graded commutes with elements in $\wedge^\bullet V$ and with the action of  $\Clif(V\oplus V^*)$. Hence the proof boils down to the computation that the untwisted Courant bracket on $TM \oplus \wedge^\bullet T^*M$ is the derived bracket associated to $d$.
\end{proof}

\begin{definition}
The bundle $(C_{V^*}, \Cour{\cdot,\cdot}_H)$ is the \emph{Clifford--Courant algebroid} associated to the extension $(V^*,d^V)$ and the closed odd-form $H \in \wedge^\bullet V^*\tensor \Omega^\bullet(M)$.
\end{definition}

\begin{example}[The Clifford--Courant algebroid of an oriented sphere bundle]

Let
\[S^{2n-1}\cdots (E,H) \to M\]
 be an oriented sphere bundle endowed with a closed $2(n+k)-1$-form. Let $\psi \in \Omega^{2n-1}(E)$ be a global angular form% with $d\psi = \pi^*c$
 . Then the transgressive complex of $E$ is given by
\[\Omega_\psi(M) = \wedge^\bullet\IP{\psi}\tensor \Omega^\bullet(M).\]
In this case the graded vector space used to form the extension is $V^* = \IP{\psi}$, which has a single generator in degree $2n-1$ and the corresponding Clifford--Courant algebroid is
\begin{align*}C_\psi & = TM \oplus (\Clif(\IP{\del_\psi,\psi})\tensor \wedge^\bullet T^*M)^{od}\\
&= TM \oplus  (\IP{\del_\psi}\tensor\wedge^{ev}T^*M)\oplus ( \IP{\psi}\tensor \wedge^{ev} T^*M) \oplus (\IP{\del_\psi \psi} \tensor \wedge^{od}T^*M)
\end{align*}
Already at this level, we see that this vector bundle has higher rank than the corresponding extended Courant algebroid \eqref{eq:extended higher}. Indeed the differences are
\begin{itemize}
\item because in the latter we are dealing with a twist of fixed degree, we can also fix the degrees of the summands instead of just splitting them into even and odd,
\item the last summand in the Clifford--Courant algebroid (corresponding to $\del_\psi \psi$) is missing in the extended Courant algebroid.
\end{itemize}
\end{example}

\subsection{Isomorphism of Clifford--Courant algebroids}

With the candidate Courant algebroids at hand we are ready to state and prove the second isomorphism theorem associated to T-duality.

\begin{theorem} \label{theo:iso2}Let $(E,H,\Psi)\to M$ and $(\hat{E},\hat{H},\hat{\Psi})\to M$ be T-dual spaces with $dF = p^*H - \hat{p}^*\hat{H}$. Then there is an isomorphism of Clifford--Courant algebroids
\[\mc{T}_F \colon  C_{\IP{\Psi}}\to C_{\IP{\hat{\Psi}}}\]
defined by the property
\begin{equation}\label{eq:T-duality compatible}
\tau_F(v \cdot \gf) = \mc{T}_F(v)\cdot \tau_F(\gf).
\end{equation}
\end{theorem}
\begin{proof}
We start by defining a linear map $A\colon C_{\IP{\Psi}} \to \End(\wedge_{\hat{\Psi}}^\bullet T^*M)$, where the codomain is the space of   endomorphisms as $\wedge^\bullet T^*M$-modules.

Let $v \in C_{\IP{\Psi}}$ be given by $X + \xi$, where $X \in TM$ and $\xi  \in \Clif(\IP{\Psi}^*\oplus \IP{\Psi})\tensor \wedge^\bullet T^*M$ and $\hat\gf \in C_{\IP{\hat{\Psi}}}$. We let 
\[ A(v)(\hat \gf) =  \tau_F(v\cdot \tau_F^{-1}(\hat \gf))-X.\]
A direct check using the expression for $\tau_F$ shows that $A(v)$ is a map of right $\wedge^\bullet T^*M$-modules and as such it corresponds to the action of an element in $ \Clif(\IP{\hat{\Psi}}^*\oplus \IP{\hat{\Psi}})\tensor \wedge^\bullet T^*M$ (in the same way $ \Clif(\IP{\hat{\Psi}}^*\oplus \IP{\hat{\Psi}}) = \End(\wedge^\bullet \IP{\hat{\Psi}})$,  $\Clif(\IP{\hat{\Psi}}^*\oplus \IP{\hat{\Psi}})\tensor \wedge^\bullet T^*M = \End(\wedge^\bullet \IP{\hat{\Psi}})\tensor  \wedge^\bullet T^*M$). Therefore there is a unique element $w \in \Clif(\IP{\hat\Psi}^*\oplus \IP{\hat\Psi})\tensor \wedge^\bullet T^*M$ such that $A(v) = w$. Hence we have
 \[\tau_F(v\cdot \tau_F^{-1}(\hat \gf)) = (X + A(v))\hat \gf \in TM \oplus \Clif(\IP{\Psi}^*\oplus \IP{\Psi})\tensor \wedge^\bullet T^*M.\] 
And we define
\[\mc{T}_F(v) =  X + A(v),\]
which by construction satisfies \eqref{eq:T-duality compatible}. Notice that  modulo 2, $v\cdot$ shifts the degree of forms by $1$ and $\tau_F$ either shifts or preserves the degree of forms, therefore  $\mc{T}_F(v)$ also shifts degree by 1 and as such lies in the odd component of the Clifford bundle, that is $\mc{T}_F(v) \in C_{\IP{\hat{\Psi}}}$.

Since the brackets on $ C_{\IP{\Psi}}$ and  $C_{\IP{\hat{\Psi}}}$ are defined solely in terms of the Clifford action and the twisted differential, we conclude that $\mc{T}_F$ is a map of Courant algebroids.
\end{proof}

\section{Examples}

Now we provide a series of examples illustrating different aspects that emerge from T-duality for transgressive bundles.

\subsection{Spherical T-duality}
\begin{example}[Spherical T-duality \cite{BEM15,BEM15b,CHU24,LSW16}] 
Recall that  two oriented sphere bundles  $S^{2n-1}\cdots (E,H)\stackrel{\pi}{\longrightarrow} M$ and $S^{2k-1}\cdots (\hat{E},\hat{H})\stackrel{\hat{\pi}}{\to} M$ over a common base $M$ endowed with closed forms of degree $2(n+k)-1$ are \emph{spherical T-duals} if there is a form $F \in \Omega^{2(n+k)-2}(E \times_M \hat{E})$ such that
\begin{equation}\label{eq:condition1}
dF = p^*H - \hat{p}^*\hat{H},
\end{equation}
\begin{equation}\label{eq:condition2}
\pi_* \circ p_* (F) \neq 0,
\end{equation}
After picking global angular forms $\psi$ and $\hat{\psi}$ of degree $2n-1$ and $2k-1$ respectively, one can pass to transgressive complexes and decompose $F$ as in Lemma \ref{lem:basic-degree zero part}:
\begin{equation}\label{eq:spherical F}
F  = p^*F_E + \lambda \psi\wedge \hat{\psi} + \hat{p}^*F_{\hat{E}},
\end{equation}
where $\lambda = \pi_*\hat{p}_*F \neq 0$. Since $F$ is quadratic on the transgressive generators we see that the nondegeneracy condition imposed for spherical T-duality implies that of transgressive T-duality. Conversely, because $F$ has degree $2(n+k-1)$ and the transgressive generators have degree $2n-1$ and $2k-1$ it also follows that transgressive T-duality implies the spherical T-duality condition. Therefore, for spherical T-duality with background form of degree $2(n+k)-1$ (including T-duality for principal circle bundles) the notion transgressive T-duality agrees with the previous notions of T-duality already present in the literature.

In \eqref{eq:spherical F} we used the decomposition of a generic T-duality kernel, $F$. Yet the terms $p^*F_E$ and $\hat{p}^*F_{\hat{E}}$ can be incorporated on the data on $E$ and $\hat{E}$ as overall symmetries, so the relevant part of the T-duality kernel is
\[F  =\lambda \psi\wedge \hat{\psi}.\]
For this kernel both T-duality maps acquire a very simple form, namely
\begin{align*}
\tau_F (\gf_0 + \psi\gf_1) &= \gf_1 + \lambda \hat \psi \gf_0,\\
 \mc{T}_F(X + \xi_0 + \del_\psi \xi_1+ \psi\xi_2 + \del_{\psi} \psi \xi_3)& = X + \xi_0 + \frac{1}{\lambda}\del_{\hat{\psi}} \xi_2+ \lambda\hat{\psi}\xi_1 +  \hat{\psi}\del_{\hat\psi} \xi_3
\end{align*}
\end{example}

\subsection{Spherical T-duality with multidegree background form}

Having forms of mixed degree has been part of the theory so far, since for $H$ of degree $2n+1$, the $d^H$-cohomology is $\Z/2n\Z$-graded. But the way that the form $H$ appears, as the curvature of the higher Courant algebroid (or of a higher gerbe) seems to hint that it should have a specific degree. Yet closed odd-forms are also curvature forms of the higher Courant algebroid $T \oplus \wedge^{od}T^*$ so one could expect a generalisation in this direction and in fact the only requirement in the definition of T-duality is that $H$ should be an odd form. Next we revisit spherical T-duality but allow for a multidegree twist. 

We consider an oriented sphere bundle $S^{2n-1}\cdots E \stackrel{\pi}{\longrightarrow} M$ endowed with a multidegree form $H$. In this situation, $[\pi_*H]$ is a collection of (possibly unrelated) even cohomology classes on $M$ and one might be tempted to use them all separately to produce a T-dual space, e.g., by using them as a collection of Chern classes of a principal $U(n)$-bundle. The problem with this idea is that the nondegeneracy condition on the T-duality kernel implies that the order of transgressive generating sets must be the same for T-dual spaces. Since $E\to M$ has a transgressive set with only one generator, the same must hold for $\hat{E}\to M$, so $\hat{E}$ must also be an oriented sphere bundle over $M$. 

The case of sphere bundles is restrictive enough to allow us to prove sharper results on existence of T-duals.

\begin{theorem}
A T-dual of $S^{2n-1}{\cdots} (E,H) \stackrel{\pi}{\longrightarrow} M$ exists if and only if $\pi_*H$ is the product of an integral class of fixed degree by an invertible element in $H^\bullet(M)$, that is,
\begin{equation}\label{eq:multidegree H condition}
[\pi_*H] = \hat{\mc{E}} \mc{B},
\end{equation}
where $\hat{\mc{E}}$ is an integral class and $\mc{B}$ is a mutidegree even cohomology class with nonvanishing zero degree term, $\mc{B}_0 \neq 0$.
\end{theorem}
\begin{proof}
For the necessity, if $\hat{E}$ is T-dual to $E$,  then
\begin{equation}\label{eq:restriction on H}
0 = [\hat{p}_*(dF)] = \hat{p}_*[p^*H - \hat{p}^*\hat{H}] =  \hat{p}_*\circ p^*[H] =   \hat{\pi}^*\circ \pi_*[H].
\end{equation}
Since by the Gysin sequence for $\hat{E}$, the cohomology classes of $M$ which pull back to trivial classes on $\hat{E}$ are multiples of the Euler class of $\hat{E}$ we conclude that if $\hat{E}$ exists,  \eqref{eq:multidegree H condition} holds.

If the Euler class of $\hat{E}$ vanishes, that is $\hat{\mc{E}} =0$, then we can take $\mc{B}=1$ and the claim about the decomposition of $\pi_*H$ holds.

To complete the proof that \eqref{eq:multidegree H condition} is a necessary condition, we need to consider the case when $\hat{\mc{E}} \neq 0$. Let $\psi$ be a global angular form for $E$ so that $d\psi = \pi^*e$ is a representative of the Euler class of $E$ and similarly for $\hat\psi$ on $\hat{E}$. In this case, from the Gysin sequence for the cohomology of $E\times_M \hat{E}$, there is no closed form in the correspondence space which integrates nonzero along the fibers of $E\times_M \hat{E} \to M$.
Therefore, if $F'$ is another form for which $dF' = p^*H -\hat{p}^*\hat{H}$, then $F-F'$ is closed and hence $(\pi_* \circ p_* (F-F'))_0 = 0$. Hence $(\pi_* \circ p_* (F))_0 \neq 0$ for any $F$ such that $dF =  p^*H -\hat{p}^*\hat{H}$. 

Changing $H$ and $H'$ by exact forms we can arrange that they lie in the corresponding Gysin complexes for $E$ and $\hat{E}$ and therefore we can also pick $F$ in the Gysin complex for $E\times_M \hat{E}$. So we have
\begin{equation}\label{eq:H looks like}
H = H_{[1]} +\psi  H_{[0]},
\end{equation}
\begin{equation}\label{eq:F looks like}
F = F_{[0]} +  \psi F_{[1]} +\hat\psi \hat{F}_{[1]} + \psi \hat\psi F_{[2]} ,
\end{equation}
where $H_{[i]}$ and $F_{[i]}$ are multidegree forms pulled back from $M$ whose parity agrees with the partiy of their indices and $(F_{[2]})_0 \neq 0$. Then
\[\hat{\pi}^*H_{[0]} = \hat{p}_* p^* H =\hat{p}_* (p^*H -\hat{p}^* \hat{H}) = \hat{p}_* dF = dF_{[1]} - \hat{e} F_{[2]} + \hat{\psi} dF_{[2]}. \]
Therefore $dF_{[2]}=0$ (in particular $(F_{[2]})_0$ is constant and nonzero) and since both sides are pulled back form $M$, we conclude that on $M$
\[H_{[0]}  = dF_{[1]} - \hat{e} F_{[2]}.\]
Passing to cohomology, we have that $\mc{B} =- [F_{[2]}]$ satisfies the conditions of the theorem. 

For the converse we assume \eqref{eq:multidegree H condition} holds. Let $\psi$ be a global angular form for $E$ so that $d\psi = \pi^*e$ is a representative of the Euler class of $E$. Let $\hat{E}$ be a sphere bundle whose Euler class is a constant multiple of $\hat{\mc{E}}$ (and after scaling and renaming we may assume it to be $\hat{\mc{E}}$).

Let $\hat{\psi}$ be a global angular form for $\hat{E}$ with $d\hat{\psi} = \hat{e}$, a representative for $\hat{\mc{E}}$. Since the inclusion of the Gysin complex is a quasi-isomorphism, after changing $H$ by an exact term (which does not afect the existence of T-duals) we can assume that $H \in \Omega^\bullet_\psi(E)$ and write
\[H = \pi^*H_ {[0]} + \psi \pi^*H_{[1]}.\]
With $H_{[0]}$ and $H_{[1]}$ multidegree forms. Since $[H_{[1]}]$ is a multiple of $[\hat{e}]$, after changing by another exact element we can arrange that
\[H = H_{[0]} + \psi \hat{e} H_{[1]}',\]
Let $\hat{H} = H_{[0]} + \hat{\psi} e H_{[1]}'$, then
\[H - \hat{H} = d(-\psi\hat{\psi} H_{[1]}')\]
We take
\[F = -\psi\hat{\psi} H_{[1]}'\]
 and observe that $[\pi_*\circ p_* F] = [H_{[1]}'] = \mc{B}$. In particular $[(\pi_*\circ p_* F)_0]=\mc{B}_0 \neq 0$.
\end{proof}

\begin{corollary}
If $S^{2n-1}\cdots (E,H) \stackrel{\pi}{\longrightarrow}M$ and $S^{2k-1}\cdots (\hat{E},\hat{H}) \stackrel{\hat{\pi}}{\longrightarrow}M$ are T-dual and
$\pi_*[H] \neq 0$, then the degree of the first nonzero component of $\pi_*[H] \neq 0$ is $2k$.
\end{corollary}
Stated in a less precise way, the dimension of the dual sphere bundle is determined by the fiber integration of $H$.

\subsection{Transgressive T-duality for frame bundles}

Now we consider the most direct applications of transgressive T-duality to frame bundles, which in many ways mimic the usual T-duality construction but now for non-Abelian gauge groups.
\begin{example}[Constructive T-duality for Hermitian vector bundles I]\label{ex:constructive 1}   Let $V\to M$ be a Hermitian vector bundle of rank $n$ and $E\to M$ be the associated frame bundle. Assume that $H \in \Omega^{2n+1}(E)$ is a closed form. Let $\Psi = \{\psi_1,\dots,\psi_{2n-1}\}$ be a transgressive generating set such that
\[\deg(\psi_{2i-1}) = 2i-1, \qquad d\psi_{2i-1} = c_i,\]
where $[c_i]$ is the $i^{th}$ Chern class of $V$ (for example, take $\psi_{2i-1}$ the forms one would get from Chern--Simons theory).

If we add two further hypothesis we can construct concrete T-duals for $(E,H)$.
\begin{itemize}
\item[(H1)] $H$ formally has only ``one leg along the fibers'', that is, after changing $H$ by an exact element, $H$ is of the form
\[H = \sum_{i=1}^n \hat{\varepsilon}_{2(n-i+1)} \wedge \psi_{2i-1} + h,\]
where each $\hat{\varepsilon}_{2(n-i+1)}$ is a basic form of degree $2(n-i+1)$ and $h$ is a basic form of degree $2n+1$.

Then, since $H$ is closed, each $\hat{\varepsilon}_{2(n-i+1)}$ is a closed form.
\item[(H2)] The classes $[\hat{\varepsilon}_{2(n-i+1)}]$ are nonzero multiples of integral classes.
\end{itemize}
Notice that (H2) holds if, for example, $H$ represents an integral class. Further, it guarantees the existence of a Hermitian vector bundle of rank $n$, $\hat{V} \to M$, whose Chern classes are multiples of $[\hat{\varepsilon}_{2(n-i+1)}]$ \cite{BV03}.

Under these hypothesis, let $\hat{E}\to M$ be the frame bundle of $\hat{V}$ and pick a transgressive generating set, $\hat{\Psi}$, for which $d\hat{\psi}_{2i-1} = \hat{c}_i = \lambda_i \hat \varepsilon_{2i}$ for some $\lambda_i \neq 0$. Endow $\hat{E}$ with the form
\[\hat{H} = \sum_{i=1}^n \frac{1}{\lambda_i}c_{n-i+1} \wedge \hat{\psi}_{2i-1} + h.\]
A direct computation shows that $\hat{H}$ is closed because $H$ is closed.

Then in the fiber product $E \times_M \hat{E}$, take as T-duality kernel
\[F = -\sum_{i=1}^{n}\frac{1}{\lambda_i}\psi_{2(n-i)+1} \wedge \hat{\psi}_{2i-1}.\]
One readily checks that  with these choices the gerbe trivialization condition holds and $F$ is quadratic and nondegenerate (the underlying matrix is diagonal with $\frac{1}{\lambda_i}$ along the diagonal). Therefore $(E,H)$ and $(\hat{E},\hat{H})$ are a T-dual pair.
\end{example}
\begin{remark}
We do no touch $K$-theoretical implications of $T$-duality in this paper, but it is worth mentioning that for those to hold, integrality conditions must be in place. In particular, the expectation is that the framework that will lead to isomorphisms of twisted $K$-theory would be that the form $H$ arises as the curvature of a higher gerbe as introduced in \cite{LSW16}, and the coefficients $\hat \varepsilon_{2i}$ of $H$ represent the Chern classes of a Hermitian vector bundle of rank $n$, that is, we can take $\lambda_i =1$ for all $i$. This is a much more restrictive situation.
\end{remark}

\begin{remark}
In the previous example, the T-dual pair consisted of two principal $U(n)$-bundles over a common base, so even though there was topology change (the Chern classes changed) the type of objects related by T-duality was the same. This does not have to be the case. One obvious change is that one could just as well use the classes $\hat{\varepsilon}_{2i}$ as Euler classes of $S^{2i-1}$-sphere bundles over $M$ and the same computation would lead to T-duality between a frame bundle and a bundle whose fiber is the product of spheres $S^1\times \dots \times S^{2n-1}$. 
\end{remark}

A more natural way to change the type of object the T-dual is arises by considering higher degree forms as twists, as we explain next.

\begin{example}[Constructive T-duality for Hermitian vector bundles II]\label{ex:constructive 2}   In this example we show that by allowing the twisting form to have degree higher than the rank of the underlying vector bundle, one obtains dualities between the full frame bundle of $V\to M$ and a partial frame bundle of a vector bundle of higher rank $\hat{V}\to M$.
  
As before,  let $V\to M$ be a Hermitian vector bundle of rank $n$ and let $E\to M$ be the associated frame bundle. Assume that $H \in \Omega^{2n+2k+1}(E)$ is a closed form and $k$ is a positive integer. Let $\Psi = \{\psi_1,\dots,\psi_{2n-1}\}$ be a transgressive generating set such that
\[\deg(\psi_{2i-1}) = 2i-1, \qquad d\psi_{2i-1} = c_i,\]
where $[c_i]$ is the $i^{th}$ Chern class of $V$.

We place again the same two extra hypotheses (H1) and (H2) on $H$. Notice that because $H$ has degree higher than $2n+1$, the coefficient of lowest degree (the coefficient of $\psi_{2n-1}$) is $\hat\varepsilon_{2(k+1)}$ which has degree greater than 2.  Similarly, the coefficient of highest degree is $\varepsilon_{2(n+k)}$, which has degree higher than the (real) rank of $V$.

Yet, once again because the forms $\hat\varepsilon_{2i}$ represent multiples of integral classes, there is a Hermitian bundle $\hat{V}\to M$ of rank $n+k$ whose Chern classes are nonzero multiples of $[\varepsilon_{2i}]$ (and we allow for the first $k$ Chern classes to be arbitrary).  We let $\hat{E}\to M$ be the partial frame bundle of $n$-unitary vectors on $\hat{V}$. Then $\hat{E}$ is an odd transgressive fibration over $M$ with transgressive generating set $\hat{\Psi} = \{\hat{\psi}_{2k+1},\dots,\hat{\psi}_{2(n+k)-1}\}$ and we can pick the forms $\hat\psi_{2i-1}$ satisfying
\[d\hat\psi_{2i-1} = \hat{c}_i= \lambda_i \hat\varepsilon_{2i}.\]
If we endow $\hat{E}$ with the form
\[\hat{H} = \sum_{i=0}^{n-1} \frac{1}{\lambda_i}c_{n-i} \wedge \hat{\psi}_{2(k+i)-1} + h,\]
then, just as in the previous example, one readily sees that $\hat{H}$ is closed and
\[F = -\sum_{i=0}^{n-1} \frac{1}{\lambda_i}\psi_{2(n-i)-1}\wedge \hat{\psi}_{2(k+i)-1}  \]
fulfils the conditions of a T-duality kernel.
\end{example}

So far we have focused on the question: Given $(E,H)\to M$, can we construct a T-dual bundle $(\hat{E},\hat{H})\to M$. Another question which might be of interest is: Given two Hermitian vector bundles $V\to M$ and $\hat{V}\to M$ is there a T-duality relation between some of their partial frame bundles? We illustrate this next   

\begin{example}\label{ex:which H}
Given two Hermitian vector bundles $V, \hat{V}\to M$ of rank $n$ and $\hat{n}$, assume that there is a quadratic relation between their Chern classes:
\begin{equation}\label{eq:relation}
\sum_{i=0}^{\hat{n}-k} \lambda_{k+i}[c_{n-i}][\hat{c}_{k+i}] =0,
\end{equation}
where $k> \hat{n}-n$ is an integer and $\lambda_i$ are nonzero real numbers\footnote{Notice that the integer $k$ here measures the difference between the degree of the top Chern class of $V$ and the degree of the relation given. Because this is not symmetric on $V$ and $\hat{V}$ the conditions that follow are not symmetric in the most obvious way in $k$.}. Let $E,\hat{E}\to M$ be the partial frame bundles of $(\hat{n}-k+1)$ unitary vectors on $V$ and $\hat{V}$, respectively. Pick transgressive generating sets such that $d\psi_{2i-1} = c_i$,  $d\hat\psi_{2i-1} = \hat{c_i}$. The relation \eqref{eq:relation} at the form level becomes
  \begin{equation}\label{eq:relation2}
\sum_{i=0}^{\hat{n}-k} \lambda_{k+i}c_{n-i}\hat{c}_{k+i} +dh =0,
\end{equation}
for some form $h \in \Omega^\bullet(M)$.

If we endow $E$ with the form $H = \sum_{i=0}^{\hat{n}-k} \lambda_{k+i}\psi_{2(n-i)-1}\hat{c}_{k+i} +h$ and $\hat{E}$ with the form $\hat{H} = \sum_{i=0}^{\hat{n}-k} \lambda_{k+i}c_{n-i}\hat{\psi}_{2(k+i)-1} +h$, we see that both $H$ and $\hat{H}$ are closed by virture of \eqref{eq:relation2} and
\[F = \sum_{i=0}^{\hat{n}-k} -\lambda_{k+i}\psi_{2(n-i)-1}\hat{\psi}_{2(k+i)-1},\]
satisfies the conditions of T-duality kernel.

A particular situation in which the hypotheses of the example are met is when $V$ and $\hat{V}$ have rank $n$ and the base manifold has dimension $2n$. In this case, purely for degree reasons, equation \eqref{eq:relation} holds for all $k>0$ and hence the (partial) frame bundles of $V$ and $\hat{V}$ are T-dual.  More generally the same happens when the degree of the relation \eqref{eq:relation} is higher than the dimension of the base manifold. Interestingly, in this situation $H$ and $\hat{H}$ may represent nontrivial cohomology classes and the exterior derivative of the $\psi_i$ can be nontrivial in cohomology.
\end{example}

\subsection{Transgressive T-duality for frame bundles with multidegree background form}

A feature of Examples \ref{ex:constructive 1} and \ref{ex:constructive 2} is that the coefficient in $H$ of the high degree transgressive elements became the lower degree Chern classes of the dual bundle. Once we allow for multi-degree forms as twists, this flipping of degrees may not occur. 

\begin{example}For an extreme case, assume that $V$ and $\hat{V}\to M$ are Hermitian vector bundles of rank $n$ and assume we have a collection of relations
\begin{equation}\label{eq:cohomological relations}
[c_i][\hat{c}_i] = 0, \quad \mbox{ for } i = n-k+1,\dots, n.
\end{equation}
Such relations might be a result of specific behaviour of the cohomology algebra of $M$ or might exist simply because the degrees of the forms in question are greater than the dimension of the manifold. Following the same lines of Example \ref{ex:which H} we pick transgressive elements of the partial unitary $k$-frame bundle of $V$ and $\hat{V}$ for which $d\psi_{2i-1} = c_i$ and write \eqref{eq:cohomological relations} at the cochain level:
\begin{equation}
c_i \hat{c}_i + dh_{4i-1} =0,\quad \mbox{ for }  i = n-k+1,\dots, n.
\end{equation}
Then let
\[H = \sum\hat{c}_i\psi_{2i-1} + h_{4i-1}, \quad \hat{H} = \sum c_i\hat\psi_{2i-1} + h_{4i-1}, \quad F = \sum -\psi_{2i-1}\hat{\psi}_{2i-1}.\]
For these choices, $(E,H)$ and $(\hat{E},\hat{H})$ are a T-dual pair with T-duality kernel $F$.

In this example, $H$ has components in degree from $2(n-k)+1$ up to $4n-1$ and the coefficients of the lower degree transgressive classes determines the lower Chern classes of the dual bundle.
\end{example}

\subsection{Non-quadratic T-duality kernel}

In all examples we considered so far, the T-duality kernel was quadratic on the transgressive generating set. For the original version of T-duality \cite{BEM04} in which the  twisting form $H$ has degree three, the T-duality kernel  has degree two and hence kernels are automatically quadratic. As a result, we often navigate towards the use of quadratic kernels.

\begin{example}[The usual version]
Let $E\to M$ be a principal $T^4$ bundle endowed with the form $H = 0 \in \Omega^5(E)$. Since $H$ vanishes, it satisfies both the hypothesis (H1) and (H2) introduced earlier (i.e. ``only one leg on the fiber'' and the different fiber integrations are multiples of an integral classes). Following the usual construction of spherical T-duals, we would be led to construct the T-dual space $\hat{E} = S^3 \times S^3 \times S^3 \times S^3 \times M$ endowed with the 5-form
\[\hat{H} = c_1 \hat{\psi}_1+c_2 \hat{\psi}_2+c_3 \hat{\psi}_3+c_4 \hat{\psi}_4,\]
where $c_i$ is the first Chern class associated to the $i^{th}$ circle in $T^4 = S^1 \times S^1\times S^1\times S^1$  and $\hat{\psi}_i$ is a closed form on $\hat{E}$ corresponding to the volume form of the $i^{th}$ $S^3$ factor.

In particular, this prescription for the construction of a T-dual space automatically generates a T-dual whose dimension is higher than the original space and does the usual exchange of cohomological content of the twisting form $H$ and the characteristic classes of $E$, leading to a trivial bundle $\hat{E}\to M$.
\end{example} 

\begin{example}[The unusual version]\label{ex:identitiy?}
Keeping the same space $(E,H)$ as before, we claim that $(E,H)$ is also self-T-dual. Indeed, since the twists vanish on either side, we are looking for a closed, degree four T-duality kernel. We take
\[F = (\psi_1-\hat{\psi}_1)\wedge(\psi_2-\hat{\psi}_2)\wedge(\psi_3-\hat{\psi}_3)\wedge(\psi_4-\hat{\psi}_4),\]
where $\psi_i$ and $\hat{\psi}_i$ are the global angular forms corresponding to the $i^{th}$ circle in both copies of $E$. Since by choice $d\psi_i = d\hat\psi_i = c_i$, the form $F$ above is closed and  from a direct computation we have:
\begin{align*}
 \tau_F (\psi_{1234}) &= 1 + \hat{\psi}_{1234}\\
 \tau_F (\psi_I)&= \hat\psi_I, \quad \quad\quad\mbox{ if } I \neq \{1,2,3,4\}.
 \end{align*}
which shows that $\tau_F$ is invertible, hence the nondegeneracy condition is satisfied. 

We can also describe the T-duality map of Courant algebroids. A simple way to get an expression is by noting that if we ignore the distinction between $\psi$ and $\hat\psi$ for a bit, then $\tau_F$ can be described by a single formula:
\[ \tau_F (\gf)= (1 + \del_{\psi_{1234}})\gf.\]
Therefore, for $v \in C_{\IP{\Psi}}$
\begin{align*}
 \tau_F (v\cdot\gf)&= (1 + \del_{\psi_{1234}})v \cdot \gf\\
 &=  (1 + \del_{\psi_{1234}})v (1 -\del_{\psi_{1234}})(1 + \del_{\psi_{1234}})\cdot \gf\\
 & =  (1 + \del_{\psi_{1234}})v (1 -\del_{\psi_{1234}}) \cdot \tau_F(\gf).
 \end{align*}
So the map of Clifford--Courant algebroids is 
\[\mc{T}_F(v) =  (1 +\del_{\psi_{1234}})v (1  -\del_{\psi_{1234}}).\] 
\end{example}

\bibliographystyle{hyperamsplain-nodash}
\bibliography{references}

\end{document}